\documentclass{article}
\UseRawInputEncoding

\usepackage{amsmath,amsthm,amssymb,amsfonts}

\usepackage{verbatim}
\usepackage{graphicx}

\usepackage{stmaryrd} %for \leftrightarroweq
\usepackage{hyperref}

\usepackage[colorinlistoftodos]{todonotes}

\newtheorem{thm}{Theorem}

\newtheorem{prop}[thm]{Proposition}

\newtheorem{assert}[thm]{Assertion}

\newtheorem{remarks}[thm]{Remark}

\newtheorem{definition}[thm]{Definition}

\newtheorem{exl}[thm]{Example}

\numberwithin{thm}{section}

\newcommand{\adj}{\leftrightarrow}

\newcommand{\adjeq}{\leftrightarroweq}

\DeclareMathOperator{\id}{id}
\DeclareMathOperator{\Fix}{Fix}

\def\Z{{\mathbb Z}}
\def\N{{\mathbb N}}

\def\R{{\mathbb R}}

\begin{document}

\title{Remarks on Fixed Point Assertions in Digital Topology, 7}

\author{Laurence Boxer
\thanks{Department of Computer and Information Sciences, Niagara University, NY 14109, USA
and  \newline
Department of Computer Science and Engineering, State University of New York at Buffalo \newline
email: boxer@niagara.edu
}
}

\date{ }
\maketitle
\begin{abstract}
This paper continues a series discussing flaws in
published assertions concerning fixed points in
digital images. \newline \newline
\end{abstract}

\section{Introduction}
As stated in~\cite{Bx19}:
\begin{quote}
The topic of fixed points in digital topology has drawn
much attention in recent papers. The quality of
discussion among these papers is uneven; 
while some assertions have been correct and interesting, others have been incorrect, incorrectly proven, or reducible to triviality.
\end{quote}
Here, we continue the work
of~\cite{BxSt19,Bx19,Bx19-3,Bx20,Bx22,BxBad6}, 
discussing many shortcomings in earlier papers 
and offering corrections and improvements.

The topic of freezing
sets~\cite{BxFPsets2} belongs to the
fixed point theory of digital topology 
and is central to the 
paper~\cite{AlmuEtAl}. We show that
the latter paper contains no
original results.

Other papers studied herein
contain assertions of fixed points
in digital metric spaces.
Quoting and paraphrasing~\cite{Bx22}:
\begin{quote}
Authors of many weak papers concerning 
fixed points in digital topology
seek to obtain results in a ``digital metric space" (see section~\ref{DigMetSp} for its definition).
This seems to be a bad idea. We slightly paraphrase~\cite{Bx20}:
\begin{quote}
\begin{itemize}
    \item Nearly all correct nontrivial published 
    assertions concerning digital
    metric spaces use the metric and do not use the
    adjacency. As a result, the digital metric space
    seems to be an artificial notion,
    not really concerned with digital
    images.
    \item If $X$ is finite (as in a ``real 
    world" digital image) or the metric $d$ 
          is a common metric such as any $\ell_p$ metric, then 
          $(X,d)$ is uniformly discrete as a topological space, 
          hence not very interesting.
    \item Many published assertions concerning
          digital metric spaces mimic analogues for subsets 
          of Euclidean~$\R^n$. Often, the authors neglect 
          important differences between the topological 
          space $\R^n$ and digital images, resulting in 
          assertions that are incorrect or incorrectly ``proven,"
          trivial, or trivial when restricted to conditions that many regard as 
          essential. E.g., in many cases, functions that
          satisfy fixed point assertions must be constant or fail to be digitally continuous~\cite{BxSt19,Bx19,Bx19-3}.
\end{itemize}
\end{quote}
\end{quote}
Since acceptance for publication
of~\cite{BxBad6}, additional
highly flawed papers rooted in digital metric spaces
have come to our attention. These 
include~\cite{DalalJAM, GayaHema, GuptaEtAl, Rani, SalJha, Sug, TiwariEtAl}.

\section{Preliminaries}
Much of the material in this section is quoted or
paraphrased from~\cite{Bx20}.

We use $\N$ to represent the natural numbers,
$\Z$ to represent the integers, $\R$ to 
represent the reals, and $\N^* = \N \cup \{0\}$.

A {\em digital image} is a pair $(X,\kappa)$, where $X \subset \Z^n$ 
for some positive integer $n$, and $\kappa$ is an adjacency relation on $X$. 
Thus, a digital image is a graph.
In order to model the ``real world," we usually take $X$ to be finite,
although there are several papers that consider
infinite digital images. The points of $X$ may be 
thought of as the ``black points" or foreground of a 
binary, monochrome ``digital picture," and the 
points of $\Z^n \setminus X$ as the ``white points"
or background of the digital picture.

\subsection{Adjacencies, 
%connectedness, 
continuity, fixed point}

In a digital image $(X,\kappa)$, if
$x,y \in X$, we use the notation
$x \adj_{\kappa}y$ to
mean $x$ and $y$ are $\kappa$-adjacent; we may write
$x \adj y$ when $\kappa$ can be understood. 
We write $x \adjeq_{\kappa}y$, or $x \adjeq y$
when $\kappa$ can be understood, to
mean 
$x \adj_{\kappa}y$ or $x=y$.

The most commonly used adjacencies in the study of digital images 
are the $c_u$ adjacencies. These are defined as follows.
\begin{definition}
Let $X \subset \Z^n$. Let $u \in \Z$, $1 \le u \le n$. Let 
$x=(x_1, \ldots, x_n),~y=(y_1,\ldots,y_n) \in X$. Then $x \adj_{c_u} y$ if 
\begin{itemize}
    \item $x \neq y$,
    \item for at most $u$ distinct indices~$i$,
    $|x_i - y_i| = 1$, and
    \item for all indices $j$ such that $|x_j - y_j| \neq 1$ we have $x_j=y_j$.
\end{itemize}
\end{definition}

\begin{definition}
\label{path}
    Let $(X,\kappa)$ be a digital image. Let
    $x,y \in X$. Suppose there is a set
    $P = \{x_i\}_{i=0}^n \subset X$ such that
$x=x_0$, $x_i \adj_{\kappa} x_{i+1}$ for
$0 \le i < n$, and $x_n=y$. Then $P$ is a
{\em $\kappa$-path} (or just a {\em path}
when $\kappa$ is understood) in $X$ from $x$ to $y$,
and $n$ is the {\em length} of this path.
\end{definition}

\begin{definition}
{\rm \cite{Rosenfeld}}
A digital image $(X,\kappa)$ is
{\em $\kappa$-connected}, or just {\em connected} when
$\kappa$ is understood, if given $x,y \in X$ there
is a $\kappa$-path in $X$ from $x$ to $y$.
\end{definition}

\begin{definition}
{\rm \cite{Rosenfeld, Bx99}}
Let $(X,\kappa)$ and $(Y,\lambda)$ be digital
images. A function $f: X \to Y$ is 
{\em $(\kappa,\lambda)$-continuous}, or
{\em $\kappa$-continuous} if $(X,\kappa)=(Y,\lambda)$, or
{\em digitally continuous} when $\kappa$ and
$\lambda$ are understood, if for every
$\kappa$-connected subset $X'$ of $X$,
$f(X')$ is a $\lambda$-connected subset of $Y$.
\end{definition}

\begin{thm}
{\rm \cite{Bx99}}
A function $f: X \to Y$ between digital images
$(X,\kappa)$ and $(Y,\lambda)$ is
$(\kappa,\lambda)$-continuous if and only if for
every $x,y \in X$, if $x \adj_{\kappa} y$ then
$f(x) \adjeq_{\lambda} f(y)$.
\end{thm}

\begin{remarks}
    For $x,y \in X$, $P = \{x_i\}_{i=0}^n \subset X$
is a $\kappa$-path from $x$ to $y$ if and only if
$f: [0,n]_{\Z} \to X$, given by $f(i)=x_i$, is
$(c_1,\kappa)$-continuous. Therefore, we may also
call such a function $f$ a {\em ($\kappa$-)path}
in $X$ from $x$ to $y$.
\end{remarks}

\begin{comment}
\begin{thm}
\label{composition}
{\rm \cite{Bx99}}
Let $f: (X, \kappa) \to (Y, \lambda)$ and
$g: (Y, \lambda) \to (Z, \mu)$ be continuous 
functions between digital images. Then
$g \circ f: (X, \kappa) \to (Z, \mu)$ is continuous.
\end{thm}
\end{comment}

We use $1_X$ to denote the identity function on $X$, 
and $C(X,\kappa)$ for the set of functions 
$f: X \to X$ that are $\kappa$-continuous.

A {\em fixed point} of a function $f: X \to X$ 
is a point $x \in X$ such that $f(x) = x$. We denote by
$\Fix(f)$ the set of fixed points of $f: X \to X$.

Let $X = \Pi_{i=1}^n X_i$. The 
{\em projection to the
$j^{th}$ coordinate}
function $p_j: X \to X_j$
is the function defined
for $x = (x_1, \ldots, x_n) \in X$, $x_i \in X_i$, by $p_j(x) = x_j$.

As a convenience, if $x$ is a point in the
domain of a function $f$, we will often
abbreviate ``$f(x)$" as ``$fx$".

\subsection{Digital metric spaces}
\label{DigMetSp}
A {\em digital metric space}~\cite{EgeKaraca15} is a triple
$(X,d,\kappa)$, where $(X,\kappa)$ is a digital image and $d$ is a metric on $X$. The
metric is usually taken to be the Euclidean
metric or some other $\ell_p$ metric; 
alternately, $d$ might be taken to be the
shortest path metric. These are defined
as follows.
\begin{itemize}
    \item Given 
          $x = (x_1, \ldots, x_n) \in \Z^n$,
          $y = (y_1, \ldots, y_n) \in \Z^n$,
          $p > 0$, $d$ is the $\ell_p$ metric
          if \[ d(x,y) =
          \left ( \sum_{i=1}^n
          \mid x_i - y_i \mid ^ p
          \right ) ^ {1/p}. \]
          Note the special cases: if $p=1$ we
          have the {\em Manhattan metric}; if
          $p=2$ we have the 
          {\em Euclidean metric}.
    \item \cite{ChartTian} If $(X,\kappa)$ is a 
          connected digital image, 
          $d$ is the {\em shortest path metric}
          if for $x,y \in X$, $d(x,y)$ is the 
          length of a shortest
          $\kappa$-path in $X$ from $x$ to $y$.
\end{itemize}

\begin{remarks}
If $X$ is finite or  
\begin{itemize}
\item {\rm \cite{Bx19-3}}
$d$ is an $\ell_p$ metric, or
\item $(X,\kappa)$ is connected and $d$ is 
the shortest path metric,
\end{itemize}
then $(X,d)$ is uniformly discrete.

For an example of a digital metric space
that is not uniformly discrete, see
Example~2.10 of~{\rm \cite{Bx20}}.
\end{remarks}

We say a sequence $\{x_n\}_{n=0}^{\infty}$ is 
{\em eventually constant} if for some $m>0$, 
$n>m$ implies $x_n=x_m$.
The notions of convergent sequence and complete digital metric space are often trivial, 
e.g., if the digital image is uniformly 
discrete, as noted in the following, a minor 
generalization of results 
of~\cite{HanBanach,BxSt19}.

\begin{prop}
\label{eventuallyConst}
{\rm \cite{Bx20}}
Let $(X,d)$ be a metric space. 
If $(X,d)$ is uniformly discrete,
then any Cauchy sequence in $X$
is eventually constant, and $(X,d)$ is a complete metric space.
\end{prop}

We will use the following.

\begin{thm}
    \label{d-under-1-const}
    Let $(X,d,\kappa)$ be a connected digital 
    metric space in which 
    \begin{itemize}
        \item $d$ is the shortest path metric, or
        \item $d$ is any $\ell_p$ metric and
              $\kappa = c_1$.
    \end{itemize}
    Let $f,g: X \to X$ be such that 
    \[ d(fx,fy) < d(gx,gy) \mbox{ for all } 
        x,y \in X.\]
    If $g \in C(X,\kappa)$, then $f$ is a constant
    function.
\end{thm}

\begin{proof}
    Let $x \adj_{\kappa} y$ in $X$. Then our
    choices of $d$ and $\kappa$, and the
    continuity of $g$, imply
    \[ d(fx,fy) < d(gx,gy) \le 1,
    \]
    so $d(fx,fy) = 0$, i.e., $fx=fy$.

    Now let $a,b \in X$. Since $X$ is connected,
    there is a $\kappa$-path $\{x_i\}_{i=0}^n$
    in $X$ from $a$ to $b$. By the above,
    $f(x_i)=f(x_{i+1})$ for $i \in \{1,\ldots,n-1\}$.
    Thus $f(a)=f(b)$. Hence $f$ is constant.
\end{proof}

\begin{comment}
Thus, the use of $c_u$-adjacency and the continuity assumption (as well as the assumption
of an $\ell_p$ metric) in the following
Proposition~\ref{shrinkage} should not be viewed as major restrictions.
The following is taken from the proof of {\bf Remark~5.2 of~\cite{BxSt19}}.
\begin{prop}
\label{shrinkage}
Let $X$ be $c_u$-connected. Let $T \in C(X,c_u)$. Let $d$ be an
$\ell_p$ metric on $X$, and $0< \alpha < \frac{1}{u^{1/p}}$.
Let $S: X \to X$ such that $d(S(x), S(y)) \le \alpha d(T(x),T(y))$ for all $x,y \in X$.
Then~$S$ must be a constant function.
\end{prop}

Similar reasoning leads to the following.

\begin{prop}
\label{converge}
{\rm \cite{Bx20}}
Let $(X,d)$ be a uniformly discrete metric 
space. Let $f \in C(X,\kappa)$. Then if $\{x_n\}_{n=1}^{\infty} \subset X$ and 
$\lim_{n \to \infty} x_n = x_0 \in X$, then for almost all $n$, $f(x_n) = f(x_0)$.
\end{prop}

Other choices of $(X,d)$ need not lead to the conclusion of Proposition~\ref{converge},
as shown by the following example.

\begin{exl}
\label{nonStdMetric}
{\rm \cite{Bx20}}
Let $X = \N \cup \{0\}$, 
\[ d(x,y) = \left \{ \begin{array}{ll}
            0 & \mbox{if } x=0=y; \\
            1/x & \mbox{if } x \neq 0 = y; \\
            1/y & \mbox{if } x = 0 \neq y; \\
            |1/x - 1/y| & \mbox{if } x \neq 0 \neq y.
            \end{array} \right .
\]
Then $d$ is a metric, and $\lim_{n \to \infty} d(n,0) = 0$. However,
the function $f(n)=n+1$ satisfies $f \in C(X,c_1)$ and
\[ \lim_{n \to \infty} d(0, f(n)) = 0 \neq f(0).
\]
\end{exl}
\end{comment}

Given a bounded metric space $(X,d)$,
the {\em diameter of} $X$ is
\[ diam(X) = \max \{ \, d(x,y) \mid
   x,y \in X \, \}.
\]

\section{\cite{AlmuEtAl} and freezing sets}
Freezing sets are defined as follows.
\begin{definition}
    {\rm \cite{BxFPsets2}}
 \label{freezeDef}
Let $(X,\kappa)$ be a
digital image. We say
$A \subset X$ is a 
{\em freezing set for $X$}
if given $g \in C(X,\kappa)$,
$A \subset \Fix(g)$ implies
$g=\id_X$.
\end{definition}

Several papers subsequent
to~\cite{BxFPsets2} have
further developed our knowledge of
freezing sets. Such a claim
cannot be made for~\cite{AlmuEtAl}.
Below, we quote verbatim (with some 
corrections noted) each assertion
presented as new
in~\cite{AlmuEtAl}, with
the assertion it mimics 
from~\cite{BxFPsets2}. 
Since~\cite{BxFPsets2} is cited in~\cite{AlmuEtAl},
the authors of~\cite{AlmuEtAl}
should have known better.

Note~\cite{AlmuEtAl} uses
``D.I" to abbreviate
``digital image".

\subsection{Theorem 2.4 of \cite{AlmuEtAl}}
Theorem 2.4 of \cite{AlmuEtAl} reads as
follows (note ``$V$ is a freezing
subset" should be ``$A$ is a freezing
subset").
\begin{quote}
    If $(U,\kappa)$ is a D.I 
    and $V$ is a freezing
subset for $U$ and 
$f: (U,\kappa) \to (V, \lambda)$ is an
isomorphism, then $f(A)$ is a
freezing set for $(V, \lambda)$.
\end{quote}

Compare with Theorem~5.3
of~\cite{BxFPsets2}:

\begin{quote}
    Let $A$ be a freezing set for
    the digital image $(X,\kappa)$ and let
$F: (X,\kappa) \to (Y, \lambda)$ 
be an isomorphism. Then $F(A)$ is a freezing set for $(Y, \lambda)$.
\end{quote}

\subsection{Theorem 2.5 of \cite{AlmuEtAl}}
Theorem 2.5 of \cite{AlmuEtAl} 
reads as follows. 

\begin{quote}
If $(U,C_{\Z}) \subset \Z^n$ is a D.I
for $z \in [1,n]$, $f \in C(U,c_z\})$, 
$\alpha, \alpha' \in U$:
$\alpha \adj_{c_z} \alpha'$ and
$p_i(f(\alpha)) \le p_i(\alpha) \le p_i(\alpha')$, then
$p_i(f(\alpha)) \le p_i(\alpha')$.
\end{quote}

Notes on Theorem 2.5 of \cite{AlmuEtAl}:
\begin{itemize}
    \item ``$(U,C_{\Z})$" should be
``$(U,c_{z})$".
    \item Each instance of ``$\le$"
should be ``$<$".
It is easy to
construct examples for which the stated conclusion
is not obtained if we allow ``$\le$"
instead of ``$<$".
\item Other errors 
appear in the 
``proof" of this assertion.
\end{itemize} 

Compare with Lemma~5.5
of~\cite{BxFPsets2}:

\begin{quote}
Let $(X,c_u)\subset \Z^n$ be a digital image, 
$1 \le u \le n$. Let $q, q' \in X$ be such that
$q \adj_{c_u} q'$.
Let $f \in C(X,c_u)$.
\begin{enumerate}
    \item If $p_i(f(q)) > p_i(q) > p_i(q')$
          then $p_i(f(q')) > p_i(q')$.
    \item If $p_i(f(q)) < p_i(q) < p_i(q')$
          then $p_i(f(q')) < p_i(q')$.
\end{enumerate}
\end{quote}

\subsection{Theorem 2.6 of \cite{AlmuEtAl}}
Theorem 2.6 of \cite{AlmuEtAl} 
reads as follows (note the ``$\alpha$" in ii. should be ``$a$").
\begin{quote}
    i. If $(U, \kappa)$ is a D.I and 
    $A \subset U$ is a retract
of $U$, then $(A, \kappa)$ has no freezing 
sets for $(U, \kappa)$.

ii. If $(U, \kappa)$ is a reducible digital 
image and $A$ is a freezing subset for $U$,
then if $a \in U$ is a reduction point of 
$u$, $\alpha \in A$.
\end{quote}

Since $U$ is a retract of $U$ via
the identity function, item i. is
incorrect as stated. If we focus
on proper subsets that are retracts, we
can compare item i) with Theorem~5.6
of~\cite{BxFPsets2}:

\begin{quote}
    Let $(X,\kappa)$ be a digital image. Let
$X'$ be a proper subset of $X$ that is a
retract of $X$. Then $X'$ does not contain
a freezing set for $(X,\kappa)$.
\end{quote}

Compare item ii. with Corollary~5.7
of~\cite{BxFPsets2}:

\begin{quote}
    Let $(X,\kappa)$ be a reducible digital image.
Let $x$ be a reduction point for $X$. Let $A$ be
a freezing set for $X$. Then $x \in A$.
\end{quote}

\subsection{Theorem 3.2 of \cite{AlmuEtAl}}
The {\em boundary of} $X \subset Z^n$
{\rm \cite{Rosenf79}} is
\[Bd(X) = \{x \in X \, | \mbox{ there exists } y \in \Z^n \setminus X \mbox{ such that } y \adj_{c_1} x\}.
\]

Theorem 3.2 of \cite{AlmuEtAl} 
reads as follows (note 
``$\forall z \in [1,n]$" should 
be ``for some $z \in [1,n]_{\Z}$").

\begin{quote}
    Let $U \subset \Z^n$ be finite, $A$ is a
    subset of $U$, $f \in C(U,c_z)$
    $\forall z \in [1,n]$. If
    $Bd(A) \subset \Fix(f)$ and
    $Bd(A)$ is a freezing set for
    $(U,c_z)$, then $A \subset \Fix(f)$.
\end{quote}

Since a superset of a freezing set is
a freezing set, the conclusion of
this assertion as written is immediate.

Proposition~5.12
of~\cite{BxFPsets2}, which has a stronger
hypothesis in not requiring $Bd(A)$ to be a
freezing set, states the following.

\begin{quote}
    Let $X \subset \Z^n$ be finite. Let $1 \le u \le n$.
Let $A \subset X$. Let $f \in C(X,c_u)$. If
$Bd(A) \subset \Fix(f)$, then $A \subset \Fix(f)$.
\end{quote}

\subsection{Theorem 3.3 of \cite{AlmuEtAl}}
Theorem 3.3 of \cite{AlmuEtAl} 
reads as follows.
\begin{quote}
    If $\Pi_{j=1}^n [0,m_j]_{\Z} \subset \Z^n$
    such that $m_j > 1$ $\forall j$ then
    $Bd(U)$ is a minimal freezing set for
    $(U,c_n)$.
\end{quote}
Presumably, $U = \Pi_{j=1}^n [0,m_j]_{\Z}$.
Also, the authors of \cite{AlmuEtAl} fail to
prove the claim of minimality.

Compare with Theorem~5.17
of~\cite{BxFPsets2}:

\begin{quote}
Let $X = \prod_{i=1}^n[0,m_i]_{\Z}  \subset \Z^n$, where $m_i>1$ for all~$i$.
Then $Bd(X)$ is a minimal freezing set for $(X,c_n)$.
\end{quote}

\subsection{Theorem 3.4 of \cite{AlmuEtAl}}
Theorem 3.4 of \cite{AlmuEtAl} 
reads as follows, where
``$\, \kappa^1,\kappa^2 \,$" should be
``$\kappa_1,\kappa_2$".

\begin{quote}
    If $(U_i,\kappa_i)$ is a set of
    D.I $\forall i \in [1,v]_{\Z}$,
    $U = \Pi_{i=1}^v U_i$
    and a subset $A$ of $U$ is a
    freezing set for
    $(U,NP_v(\kappa^1, \kappa^2, \ldots, \kappa_v))$, then we have $p_i(A)$ is a freezing
set for $(U_i,\kappa_i)$
$\forall i \in [1,v]_{\Z}$.
\end{quote}

Compare with Theorem~5.18
of~\cite{BxFPsets2}:

\begin{quote}
        Let $(X_i,\kappa_i)$ be a
digital image, $i \in [1,v]_{\Z}$. Let 
$X = \prod_{i=1}^v X_i$.
Let $A \subset X$. Suppose
$A$ is a freezing set
for $(X,NP_v(\kappa_1,\ldots,\kappa_v))$. Then
for each $i \in [1,v]_{\Z}$,
$p_i(A)$ is a freezing set
for $(X_i,\kappa_i)$.
\end{quote}

\section{\cite{DalalJAM}'s common fixed
point results}
S. Dalal is the author or coauthor of
three papers with the title
``Common Fixed Point Results for Weakly 
Compatible Map in Digital Metric
Spaces"~\cite{DalalJAM,DalalSJPMS,DalalEtAl}.
We have discussed flaws and improvements
of~\cite{DalalSJPMS} in~\cite{Bx19-3}, and
those of~\cite{DalalEtAl} 
in~\cite{Bx19}.
In this section, we discuss flaws and 
improvements of~\cite{DalalJAM}.

\begin{definition}
\label{DalalMaps}
{\rm ~\cite{DalalEtAl}}
Suppose $S$ and $T$ are self-maps on a digital metric space $(X,d,\kappa)$. Suppose
$\{x_n\}_{n=1}^{\infty}$ is a sequence in $X$ such that
\begin{equation}
\label{seqEqn}
lim_{n \to \infty} S(x_n) = lim_{n \to \infty} T(x_n) = t \mbox{ for some } t \in X.
\end{equation}
We have the following.
\begin{itemize}
%\item $S$ and $T$ are called {\em commutative} 
%      if $S(x)=T(x)$ for all $x \in X$.
\item $S$ and $T$ are called {\em compatible} if 
      $lim_{n \to \infty} d(S(T(x_n)), T(S(x_n))) = 0$ for all
      sequences $\{x_n\}_{n=1}^{\infty} \subset X$ that satisfy
      statement~(\ref{seqEqn}).
\item $S$ and $T$ are called {\em compatible of type (A)} if 
\[lim_{n \to \infty} d(S(T(x_n)), T(T(x_n))) = 0 = \]
\[ lim_{n \to \infty} d(T(S(x_n)), S(S(x_n)))\]
for all sequences $\{x_n\}_{n=1}^{\infty} \subset X$ that satisfy
      statement~(\ref{seqEqn}).
\item $S$ and $T$ are called {\em compatible of type (P)} if
\[ lim_{n \to \infty} d(S(S(x_n)), T(T(x_n))) = 0\]
for all sequences $\{x_n\}_{n=1}^{\infty} \subset X$ that satisfy
      statement~(\ref{seqEqn}).
\end{itemize}
\end{definition}

\begin{prop}
    {\rm \cite{DalalEtAl}}
    \label{compatEquiv}
    Let $S$ and $T$ be compatible maps 
    of type (A) on a digital metric 
    space $(X,d,\rho)$. If one of
    $S$ and $T$ is continuous, then
    $S$ and $T$ are compatible.
\end{prop}

The continuity assumption of
Proposition~\ref{compatEquiv} is
of the $\varepsilon - \delta$ type of
analysis. It is often unnecessary.
Thus, we have the following.

\begin{thm}
\label{compatEquivs}
{\rm \cite{Bx19}}
Let $(X,d,\kappa)$ be a digital metric space, where either $X$ is finite 
or $d$ is an $\ell_p$ metric. Let $S$ and $T$ be self-maps on $X$. 
Then the following are equivalent.
\begin{itemize}
%\item $S$ and $T$ are commutative.
\item $S$ and $T$ are compatible.
\item $S$ and $T$ are compatible of type (A).
\item $S$ and $T$ are compatible of type (P).
\end{itemize}
\end{thm}

Indeed, other notions defined as
variants on compatibility are also
equivalent to compatibility
(see Theorem~5.6 of~\cite{BxBad6}).

\begin{prop}
    {\rm \cite{DalalJAM}}
    \label{compatLims}
    Let $S$ and $T$ be compatible maps 
    on a digital metric space
    $(X,d,\rho)$ into itself. Suppose
    $\lim_{n \to \infty} Sx_n = 
    \lim_{n \to \infty} Tx_n = t$
    for some $t \in X$. Then

    (a) $\lim_{n \to \infty} STx_n =
    Tt$ if $T$ is continuous at $t$.

    (b) $\lim_{n \to \infty} TSx_n =
    St$ if $S$ is continuous at $t$.
\end{prop}

As above, the continuity assumption of
Proposition~\ref{compatLims} is
of the $\varepsilon - \delta$ type of
analysis and is often unnecessary.
Thus, we have the following.

\begin{prop}
    \label{betterCompatLims}
    Let $S$ and $T$ be compatible maps 
    on a digital metric space
    $(X,d,\rho)$ into itself. 
    Let $(X,d)$ be uniformly
discrete. Suppose
    $\lim_{n \to \infty} Sx_n = 
    \lim_{n \to \infty} Tx_n = t$
    for some $t \in X$. Then

    {\rm (a)} $\lim_{n \to \infty} STx_n =
    Tt$.

    {\rm (b)} $\lim_{n \to \infty} TSx_n =
    St$.
\end{prop}

\begin{proof}
    By uniform discreteness, we have,
    for almost all~$n$, 
    $Sx_n=Tx_n=t$. Thus for 
    almost all~$n$,
    \[ STx_n =  
    \mbox{ (by compatibility) } TSx_n
       = Tt,\]
    which proves (a). Assertion~(b)
    follows similarly.
\end{proof}

The following appears as
Theorem~3.1 of~\cite{DalalJAM}.

\begin{assert}
\label{DalCompatibleAssert}
    Let $A$, $B$, $S$ and $T$ be four 
    self-mappings of a complete 
    digital metric space $(X,d,\rho)$
    satisfying the following 
    conditions.

    (a) $S(X) \subset B(X)$ and
        $T(X) \subset A(X)$;

    (b) the pairs $(A,S)$ and $(B,T)$
        are compatible;

    (c) one of $S,T,A$, and $B$ is
        continuous;

    (d) $d(Sx,Tx) \le
     \phi( \max \{ d(Ax,By),
     d(Sx,Ax), d(Sx,By)\})$ for all
     $x,y \in X$, where
     $\phi: [0,\infty) \to [0,\infty)$
     is continuous, monotone increasing,
     and $\phi(t) < t$ for $t > 0$.
     Then $A,B,S$, and $T$ have a
     unique common fixed point in $X$.
\end{assert}

However, the argument offered as a
proof for 
Assertion~\ref{DalCompatibleAssert}
in~\cite{DalalJAM} is marred by
an error that also appears 
in~\cite{DalalSJPMS}: A sequence
$\{y_n\}_{n=0}^{\infty} \subset X$ of
points is constructed such that
$lim_{n \to \infty} d(y_{2n},y_{2n+1})
=0$. It is wrongly concluded that
$\{y_n\}_{n=0}^{\infty}$ is a
Cauchy sequence; a counterexample
is given at Example~6.2
of~\cite{Bx19-3}.

We must conclude that
Assertion~\ref{DalCompatibleAssert}
is unproven.

\section{\cite{GayaHema}'s path-length
metric assertions}
In this section, we discuss flaws
in the paper~\cite{GayaHema}.

\subsection{Unoriginal assertions}
Several of the assertions presented
as original appear in earlier
literature.

Theorems 3.1, 3.2, and 3.3
of~\cite{GayaHema} duplicate
results of~\cite{HanBanach}, a paper
cited in~\cite{GayaHema}. 
While~\cite{HanBanach} uses 
the Euclidean metric, \cite{GayaHema}
should have noted that the proofs
of~\cite{HanBanach} also work for
their respective analogs using the
shortest-path metric.

Theorem~3.1 of~\cite{GayaHema} states
the following.
\begin{quote}
    In a digital metric space 
    $(Y^*,d)$, if a sequence
    $\{x_n\}_{n=1}^{\infty}$
    is a Cauchy sequence, then
    $x_n=x_m$ for all
    $m,n > \alpha$, where 
    $\alpha \in \N$.
\end{quote}
But this duplicates Proposition~3.5
of~\cite{HanBanach}.\newline 

Theorem~3.2 of~\cite{GayaHema} states
the following.
\begin{quote}
    In a digital metric space
    $(Y^*,d)$, if a sequence
    $\{x_n\}_{n=1}^{\infty}$
    converges to a limit
    $L \in Y^*$, then there is an
    $\alpha \in \N$ such that for all
    $m,n > \alpha$, $x_n = L$ i.e.
    $x_n=x_{n+1}=x_{n+2} = \cdots = L$
    for $n \ge \alpha$.
\end{quote}
But this duplicates Proposition~3.9
of~\cite{HanBanach}. \newline

Theorem 3.3 of~\cite{GayaHema} states
the following.
\begin{quote}
    A digital metric space 
    $(Y^*,d)$ is complete.
\end{quote}
But this duplicates Theorem~3.11
of~\cite{HanBanach}.

\subsection{Contractions}

The following definition is 
not attributed to a source
in~\cite{GayaHema}. It appears
in~\cite{EgeKaraca-Ban}, which
was cited in~\cite{GayaHema}.

\begin{definition}
\label{contraction}
    A self mapping $S$ on a digital
    metric space $(Y,d)$ is said to 
    be a {\em digital contraction mapping} 
    if and only if there exists a 
    non-negative number $q < 1$
    such that
    \[ d(Sx,Sy) \le q d(x,y)
    ~\forall ~x,y \in Y.
    \]
\end{definition}

The following is stated as
Corollary~3.1 of~\cite{GayaHema}.
\begin{assert}
\label{GH3.1}
    Let $(Y,d)$ be a digital metric space endowed with graph
    $G$, where $d$ is the path length metric. Let
    $S: Y \to Y$ 
    be a digital contraction map on
    $Y$. Then $S$ has a unique
    fixed point.
\end{assert}

\begin{remarks}
Without the assumption 
    of connectedness, the path length 
    metric is undefined, so 
    Assertion~\ref{GH3.1} as written
    is, at best, misleading.   
\end{remarks}

\begin{prop}
    If in Assertion~\ref{GH3.1} $G$
    is connected, then $S$ must be a 
    constant map.
\end{prop}

\begin{proof}
    Since $d$ is the path length metric,
    for $x \adj y$, we have
    \[ d(Sx,Sy) \le qd(x,y) = q < 1
       = d(\id_Y(x),\id_Y(y)).    
    \]
    The assertion follows from
    Theorem~\ref{d-under-1-const}.
\end{proof}

\subsection{Quasi-contractions}
\begin{definition}
    \label{quasicontraction}
    {\rm \cite{{Circ}}}
    A self mapping $S$ on a
    metric space $(Y,d)$ is said to 
    be a {\em quasi-contraction} 
    if and only if there exists a 
    non-negative number $q < 1$
    such that
    \[ \begin{array}{c}
        d(Sx,Sy) \le
    q \max \{ d(x,y), d(x,Sx), d(y,Sy),
    d(x,Sy), d(y,Sx)\} \\
     \forall ~x,y \in Y
    \end{array}
    \]
\end{definition}

Theorem~3.4 of~\cite{GayaHema} is stated
as follows.
\begin{thm}
\label{GH3.4}
    Let $(Y,d,\kappa)$ be a digital
    metric space, where $d$ is the
    shortest path metric. Let
    $S: Y \to Y$ be a quasi-contraction.
    Then \newline
    (a) for all $x \in Y$,
        $\lim_{i \to \infty} S^i x = u_1 \in Y$; \newline
    (b) $u_1$ is the unique fixed point
        of~$S$, and \newline
    (c) $d(S^i x, u_1) \le \frac{q^i}{1-q}
    d(x,Sx)$ for all $x \in Y$,
    $i \in \N$, where $q$ is as in
    Definition~\ref{quasicontraction}.   
\end{thm}

As noted above, Theorem~\ref{GH3.4}
is only valid when
$(X,\kappa)$ is connected. With the
inclusion of such a hypothesis, we show
below how Theorem~\ref{GH3.4} can
be strengthened.

\begin{thm}
    \label{GH3.4simplified}
    Let $(Y,d,\kappa)$ be a connected digital metric space, where $d$
    is the shortest path metric. Let
    $S: Y \to Y$ be a quasi-contraction.
    Then there exists $u \in Y$ such that \newline
    (a) for every $x \in Y$ there exists
        $n_0 \in \N$ such that
        $i \ge n_0$ implies
        $S^i x = u$; and \newline
    (b) $u$ is the unique fixed point
        of~$S$. 
\end{thm}

\begin{proof}
Since $(Y,d)$ is uniformly discrete,
(a) of Theorem~\ref{GH3.4} implies
$S^i x = u$ for almost all~$i$. (b) follows
immediately.
\end{proof}

\section{$\theta$-contraction assertion
of~\cite{GuptaEtAl}}

The following set of functions $\Theta$ is
defined in~\cite{GuptaEtAl}.
$\theta \in \Theta$ if
$\theta: [0,\infty) \to [0,\infty)$ and
\begin{itemize}
    \item $\theta$ is increasing;
    \item $\theta(0)=0$; and
    \item $t > 0$ implies
          $0 < \theta(t) < \sqrt{t}$.
\end{itemize}

\begin{definition}
    {\rm \cite{GuptaEtAl}}
    Let $(X,d)$ be a metric space.
    Let $T: X \to X$. 
    Let $\theta \in \Theta$.
    If $d(Tx,Ty) \le \theta(d(x,y))$ for
    all $x,y \in X$, $T$ is a {\em digital
    $\theta$-contraction}.
\end{definition}

The following is stated as the main result,
Theorem~3.1, of~\cite{GuptaEtAl}.

\begin{assert}
\label{GuptaAssert}
    Suppose $(X,d,\ell)$ is a digital
    metric space, $\theta \in \Theta$,
    and $T: X \to X$ is a digital
    $\theta$-contraction. Then
    $T$ has a unique fixed point.
\end{assert}

However, the argument offered as
a proof of this assertion is
flawed as follows. A sequence
$\{x_n\}_{n=1}^{\infty} \subset X$
is constructed such that
$\{d(x_{n+1},x_n)\}_{n=1}^{\infty}$
is a strictly decreasing sequence.
The authors conclude that
$x_n = x_{n+1}$ for large~$n$. But
this does not follow, since it
has not been shown that
$\{d(x_{n+1},x_n)\}_{n=1}^{\infty}$
decreases to~0.

We must conclude that
Assertion~\ref{GuptaAssert} is
unproven.

We note the following case, in which
Assertion~\ref{GuptaAssert} reduces to triviality.

\begin{prop}
    Let $(X,d,\kappa)$ be a 
    connected digital metric
    space in which
    \begin{itemize}
        \item $d$ is the shortest path metric, or
        \item $d$ is any $\ell_p$ metric and
              $\kappa = c_1$.
    \end{itemize} 
 Then every $\theta$-contraction
    on $(X,d)$ is a constant map.
\end{prop}

\begin{proof}
    Suppose $T$ is a
    $\theta$-contraction on $(X,d)$.
    Given $\kappa$-adjacent
    $x_0,y_0 \in X$,
    \[ d(Tx_0,Ty_0) \le 
    \theta(d(x_0,y_0)) <
    \sqrt{d(x_0,y_0)} = 1. 
    \]
Hence $d(Tx_0,Ty_0)=0$. Since $X$ is 
$\kappa$-connected, the assertion follows.
\end{proof}

\section{\cite{Rani}'s contractions}
We find the following stated as
Theorem~3.1 of~\cite{Rani}.

\begin{assert}
\label{Rani3.1}
        Let $(X,d,k)$ be a complete 
    digital metric space with 
    $k$-adjacency where $d$ is 
    usual Euclidean metric for
    $\Z^n$ and let $f$ and $g$
    be self-mappings on $X$ satisfying the following conditions:
\begin{itemize}
    \item $f(X) \subseteq g(X)$;
     \item $g$ is continuous; and
     \item for some $q$ such that
     \begin{equation}
         \label{Rani3.1Ineq}
          0 < q < 1,
         \mbox{ and for every } x, y \in X,~~
         d(fx, fy) \le q d(gx, gy).
     \end{equation}
\end{itemize}

    Then $f$ and $g$ have a unique 
    common fixed point in $X$ provided $f$ and $g$ commute.
\end{assert}

Notes:
\begin{itemize}
    \item There is an 
        error, perhaps a typo, in the argument given as a
        proof for this assertion:
        \[``d(y_n,y_m) \to 1" \mbox{ should be }
        ``d(y_n,y_m) \to 0".
        \]
    \item The continuity assumed in
    (3.2) is of the $\varepsilon - \delta$ variety of analysis.
    Since $d$ is the Euclidean
    metric in $\Z^n$, all self-maps
    on the uniformly discrete
    $(\Z^n,d)$ or any of its subsets
    have this continuity. Therefore,
    we can improve on the assertion
    of~\cite{Rani} as follows.
\end{itemize}

 \begin{thm}
 \label{betterRani3.1}
         Let $(X,d)$ be a 
         uniformly discrete metric
         space and let $f$ and $g$
    be self-mappings on $X$ 
    such that:
\begin{itemize}
    \item $f(X) \subseteq g(X)$;
    \item For some $q$ such that
    $0 < q < 1$ and every 
       $x, y \in X$,
    $d(fx, fy) \le
       q d(gx, gy)$.
    \item $f$ and $g$ commute.
\end{itemize}
Then $f$ and $g$ have a unique 
    common fixed point in $X$.
 \end{thm}

\begin{proof}
We use ideas from~\cite{Rani}.
Let $x_0 \in X$. Since
$f(X) \subseteq g(X)$, take
$x_1$ such that $fx_0 = gx_1$,
and inductively,
$fx_n = gx_{n+1}$. Then
\[ d(fx_n,fx_{n+1}) \le
   q d(gx_n,gx_{n+1})
   = q d(fx_{n-1},fx_n).\]
An easy induction yields that
\[ d(fx_n, fx_{n+1}) \le
q^n d(fx_0, fx_1)
\to_{n \to \infty} 0.
\]
Thus, there exists $z \in X$
such that for almost all~$m,n$,
\[gx_{m+1} = fx_m = fx_n = gx_{n+1}=z.
\]
Using the uniformly discrete and commutative 
properties, for almost all n,
\[ d(gz,z) = d(gfx_n, fx_n) =
d(fgx_n,fx_n) \le
q d(ggx_n,gx_n) = q d(gz,z).
\]
Thus, $d(gz,z) = 0$, so $z$ is a
fixed point of $g$.

Then
\[ d(fz,z) = d(fz,fx_n) \le
   q d(gz,gx_n) = q d(gz,z) = 0,
\]
so $z$ is a common fixed point
of $f$ and $g$.

To show the uniqueness of $z$,
suppose $z_1$ is a common fixed
point of $f$ and $g$. Then
\[ d(z,z_1) = d(fz,fz_1) \le 
   q d(gz,gz_1) = q d(z,z_1),
\]
which implies $d(z,z_1)=0$, i.e., $z = z_1$.
\end{proof}

The following provides an important case in which
Theorem~\ref{betterRani3.1}, and therefore
Assertion~\ref{Rani3.1}, reduce to triviality.

\begin{prop}
    Let $(X,d,\kappa)$ be a connected digital 
    metric space, where 
    \begin{itemize}
      \item $d$ is the shortest path metric, or
        \item $\kappa = c_1$ and $d$ is an $\ell_p$
              metric.
    \end{itemize}
    Suppose $f$ and $g$ are maps 
    satisfying~(\ref{Rani3.1Ineq}).
    If $g$ is $\kappa$-continuous, then
    $f$ is a constant function.
\end{prop}

\begin{proof}
    Let $x \adj_{\kappa} y$ in $X$.
    By~(\ref{Rani3.1Ineq}), continuity,
    and our choices of $d$ and $\kappa$,
    \[ d(fx,fy) \le qd(gx,gy) \le q < 1.
    \]
    Hence $fx=fy$. Since $X$ is
    connected, it follows as in the
    proof of 
    Theorem~\ref{d-under-1-const}
    that $f$ is constant.
\end{proof}

\section{Weakly compatible mappings in~\cite{SalJha}}
Several papers, including~\cite{SalJha},
attribute the following definition to
alleged sources that do not contain it
or, by virtue of their own citations,
clearly are not the source.

\begin{definition}
Let $(X,d)$ be a metric space
and $f,g: X \to X$. We say
$f$ and $g$ are {\em weakly
compatible} if they commute
at coincidence points, i.e.,
$f(x)=g(x)$ implies
$f(g(x)) = g(f(x))$.
\end{definition}

The following is stated as
Proposition~2.15
of~\cite{SalJha}.

\begin{assert}
\label{SJ2.15}
    Let $J,K: F \to F$ be
    weakly compatible maps.
    If a point~$\eta$ is a 
    unique point of
coincidence of mappings
$J$ and $K$, i.e.,
$J(\sigma)=K(\sigma)=\eta$,
then $\eta$ is the unique 
common fixed point of $J$ 
and $K$.
\end{assert}

The argument given to prove
Assertion~\ref{SJ2.15}
claims that 
$J(\sigma)=K(\sigma)=\eta$ 
implies $J(\sigma)=J(K(\sigma))$,
and $K(J(\sigma))=K(\sigma)$.
No reason is given to support
the latter equations, and there
is no obvious reason to
accept them. Therefore,
we must regard 
Assertion~\ref{SJ2.15} as
unproven.

The following is stated as
Theorem~3.1 of~\cite{SalJha}.

\begin{assert}
    \label{SJ3.1}
    Let $(F,d,Y)$ be a
    digital metric space, where
    $Y$ is an adjacency and
    $d$ is the
    Euclidean metric
    on~$\Z^n$. Let
    $J,K: F \to F$
    such that
    \begin{itemize}
        \item $J(F) \subset K(F)$, and
        \item for some $\xi$
        such that $0 < \xi < 1/4$,
        \[ d(Ju,Jq) \le
        \xi [d(Ju,Kq) +
         d(Jq,Ku)+
         d(Ju,Ku) +
         d(Jq,Kq)]~ \forall
         u,q \in F.
        \]
    \end{itemize}
    If $K(F)$ is complete
    and $J$ and $F$ are
    weakly compatible, then
    there exists a unique common fixed point in
    $F$ for $J$ and $K$.   
\end{assert}

The argument offered as
proof of 
Assertion~\ref{SJ3.1}
depends on
Assertion~\ref{SJ2.15},
which, we have shown above,
is unproven. Thus, we must
regard Assertion~\ref{SJ3.1}
as unproven.

\begin{remarks}
    Example~3.2 of~\cite{SalJha}
    asks us to consider as a digital
    metric space $(F,\phi,Y)$ where
    $F = [0,1]$, and the function
    $J: F \to F$ given by
    $J(u)= \frac{1}{1+u}$. But
    $F$ clearly is not a subset of
    any~$\Z^n$, and $J$ is not
    integer-valued.
\end{remarks}

\section{\cite{Sug}'s common fixed
point assertions}
\subsection{\cite{Sug}'s Theorem 3.1}
The following is stated
as Theorem~3.1
of~\cite{Sug}.

\begin{assert}
\label{Sug3.1}
    Let $(X,\ell,d)$ be a
    digital metric space.
    Let $A,B: X \to X$
    with $B(X) \subset A(X)$.
    Let $\gamma$ be a
    right continuous real
    function such that
    $\gamma(a) < a$ for
    $a > 0$. Suppose for
    all $x,y \in X$ we
    have
    \[ d(Bx,By) \le 
       \gamma(d(Ax,Ay)).\]
    Then $A$ and $B$ have a unique common fixed point.
\end{assert}

That Assertion~\ref{Sug3.1}
is incorrect is shown by the following.

\begin{exl}
    Let $X = \N$ and let
    $d(x,y) = ~ \mid x - y \mid$. Let $A(x)=x+1$,
    $B(x)=2$, $\gamma(x) = x/2$ for all
    $x \in \N$. Clearly,
    the hypotheses of
    Assertion~\ref{Sug3.1}
    are satisfied, but $A$
    has no fixed point.
\end{exl}

\subsection{\cite{Sug}'s Theorem 3.2}

\begin{definition}
    (Incorrectly attributed in~\cite{Sug}
     to~\cite{Rosenf79}; found 
     in~\cite{Al-ThagEtAl})

     Let $f,g: X \to X$ be functions.
     If there is a coincidence point $x_0$
     of $f$ and $g$ at which $f$ and $g$
     commute (i.e., $f(x_0)=g(x_0)$ and
     $f(g(x_0)) = g(f(x_0))$, then
     $f$ and $g$ are {\em occasionally
     weakly compatible}.
\end{definition}

Let $\Phi$ be the set of functions
$\phi: [0,\infty) \to [0,\infty)$ such
that $\phi$ is increasing, $\phi(t) < t$
for $t > 0$, and $\phi(0)=0$.

\begin{definition}
    {\rm \cite{SrideviEtAl}}
    Let $(X,d)$ be a metric space.
    Let $\alpha: X \times X \to [0,\infty)$.
    Let $\phi, \psi \in \Phi$. Let
    $T: X \to X$. If for
    all $x,y \in X$ we have
    \[ \alpha(x,y) \psi(d(Tx,Ty)) \le
      \psi(d(Tx,Ty)) - \phi(d(Tx,Ty))
    \]
    then $T$ is an $\alpha - \psi - \phi$
    {\em contractive mapping}.
\end{definition}

\begin{definition}
    {\rm \cite{SametEtAl}}
    \label{alphaAdmiss}
    Let $(X,d)$ be a metric space.
    Let $T: X \to X$ and
    $\alpha: X \times X \to [0,\infty)$.
    We say $T$ is {\em $\alpha$-admissible}
    if $\alpha(x,y) \ge 1$ implies
    $\alpha(Tx,Ty) \ge 1$.
\end{definition}

The following is stated as
Theorem~3.2 of~\cite{Sug}.

\begin{assert}
    \label{Sug3.2}
    Let $(X,\ell,d)$ be a digital metric
    space. Let $S,T,A$, and $B$ be
    $\alpha - \psi - \phi$
    contractive mappings of $(X,d)$.
    Let the pairs $(A, S)$ and
$(B, T)$ be occasionally weakly compatible.
Suppose for all $x,y \in X$,
\[ \alpha(x,y) \psi(d(Ax,By)) \le
   \psi(M(x,y)) - \gamma(M(x,y))
\]
    where
\[ M(x,y) = \max \left \{ \begin{array}{c}
    d(Sx,Ty),~ d(By,Sx),~d(Sx,Ax),~d(By,Ty),~
    d(Ax,Ty), \\
    \frac{2d(Sx,Ax)}{1+d(By,Ty)}
   \end{array} \right \}
\]
Then there is a unique fixed point of
$S,T,A$, and $B$.
\end{assert}

The argument offered as proof of
Assertion~\ref{Sug3.2} in~\cite{Sug}
is marred by the following flaws.
\begin{itemize}
    \item The ``$\gamma$" in the inequality
        should be ``$\phi$".
    \item The argument starts:
    \begin{quote}
        Since $A,B$ are $\alpha$-admissible
        then for all $x,y \in X$,
        $\alpha(x,y) \ge 1$.        
    \end{quote}
    Notice it was not hypothesized that
    $A$ and $B$ are $\alpha$-admissible.
    Even if this is a mere omission,
    the conclusion would be unsupported; it
    does not follow from 
    Definition~\ref{alphaAdmiss}.
    The unproven allegation that
$\alpha(x,y) \ge 1$ is part of the
argument that $Ax=Sx=By=Ty$, which in
turn is an important part of the
uniqueness argument.
\item No proof is offered for the
      claim of a fixed point for any
      of $S,T,A$, and $B$.
\end{itemize}
We must conclude that Assertion~\ref{Sug3.2}
is unproven.

\subsection{\cite{Sug}'s Example 3.2}
Example~3.2 of~\cite{Sug} wants us
to consider the digital metric space
$(\N,d,4)$, where 
$d(x,y) =~ \mid x-y \mid$. The
paper fails to define 
4-adjacency on~$\N$;
perhaps 2-adjacency was
intended.

Further flaws:
\begin{itemize}
    \item It is claimed that functions $A$, $B$,
    $S$, and $T$ have a common fixed point, where
    \[Ax = x + 1,~~ By = y + 1, ~~ Sx = x - 1,~~
    Ty = y - 1 .
    \]
    But clearly none of these functions has
    a fixed point.
    \item We are asked to consider
an inequality that uses functions
$\psi,~ \alpha$, and $\varphi$ that are
not defined.
\end{itemize}

\subsection{\cite{Sug}'s $M_6$}
In~\cite{Sug}, $M_6$ is defined as the set
of real-valued continuous functions
$\phi: [0,1]^6 \to \R$ such that

(A) $\int_0^{\phi(u.u.0,0,u,u)} \varphi(t) dt \le 0$
    implies $u \ge 0$.

(B) $\int_0^{\phi(u.u.0,0,u,0)} \varphi(t) dt \le 0$
    implies $u \ge 0$.

(C) $\int_0^{\phi(0,u.u.0,0,u)} \varphi(t) dt \le 0$
    implies $u \ge 0$.

Notice the use of two different ``phi" 
symbols, ``$\phi$" and ``$\varphi$".
\begin{itemize}
    \item If this is intended, ``$\varphi$"
          is undefined.
    \item If it is intended that ``$\varphi$"
          should be ``$\phi$", then every
          continuous function
$\phi: [0,1]^6 \to \R$ belongs to $M_6$,
since the domain of such a function
requires $u$, as a parameter of $\phi$,
to be nonnegative.          
\end{itemize}

Example~3.3 of~\cite{Sug} claims
a certain function~$\phi$ satisfies
(A) and (B) of the definition of~$M_6$
and therefore belongs to~$M_6$,
although no claim is made that
$\phi$ satisfies (C).

\section{\cite{TiwariEtAl}'s common fixed point assertions}
The paper~\cite{TiwariEtAl}
presents five assertions concerning
pairs of various types of expansive 
self-mappings on digital images.
Each of these assertions concludes
that the maps of the pair
have common fixed points.
We show below that all of these
assertions, ``usually"
or always, reduce to
triviality: further, three of them must be
regarded as unproven in full generality due to errors
in their ``proofs".

We note that~\cite{TiwariEtAl}
uses ``$\alpha$-adjacent" for
what we have been calling
``$c_{\alpha}$-adjacent".

\subsection{\cite{TiwariEtAl}
Theorem 3.1}
\label{TiwariSec1}
The assertion labeled
in~\cite{TiwariEtAl} as
Theorem~3.3.1 is clearly
intended to be labeled
Theorem~3.1. It is stated
as follows.
\begin{thm}
    {\rm \cite{TiwariEtAl}}
    \label{Tiwari3.1}
    Let $(X,d,k)$ be a 
    complete digital metric 
    space and suppose
    $T_1,T_2: X \to X$ are
    continuous, onto mappings
    satisfying
    \begin{equation}
    \label{Tiwari3.1ineq}
        d(T_1x,T_2y) \ge
       \alpha d(x,y) +
       \beta [d(x,T_1x) + d(y,T_2y)]
    \end{equation} 
    for all $x,y \in X$,
    where $\alpha > 0$,
    $1/2 \le \beta \le 1$, and
    $\alpha + \beta > 1$. Then
    $T_1$ and $T_2$ have a 
    common fixed point in~$X$.
\end{thm}

See section~\ref{modificationToTriviality}
for discussion of the triviality of this
assertion.

\subsection{\cite{TiwariEtAl}
Theorem 3.2}
Theorem~3.2 of~\cite{TiwariEtAl}
is stated as follows.

\begin{thm}
    {\rm \cite{TiwariEtAl}}
    \label{Tiwari3.2}
    Let $(X,d,k)$ be a 
    complete digital metric 
    space and suppose
    $T_1,T_2: X \to X$ are
    continuous, onto mappings
    satisfying
    \begin{equation}
    \label{Tiwari3.2ineq}
        d(T_1x,T_2y) \ge
       \alpha d(x,y) +
       \beta [d(x,T_2y) + d(y,T_1x)]
    \end{equation} 
    for all $x,y \in X$,
    where $\alpha > 0$,
    $1/2 \le \beta \le 1$, and
    $\alpha + \beta > 1$. Then
    $T_1$ and $T_2$ have a 
    common fixed point in~$X$.
\end{thm}

See section~\ref{modificationToTriviality}
for discussion of the triviality of this
assertion.

\subsection{\cite{TiwariEtAl}
``Theorem" 3.3}
The following is stated as Theorem~3.3 
of~\cite{TiwariEtAl}.

\begin{assert}
    {\rm \cite{TiwariEtAl}}
    \label{Tiwari3.3}
    Let $(X,d,k)$ be a 
    complete digital metric 
    space and suppose
    $T_1,T_2: X \to X$ are
    continuous, onto mappings
    satisfying
    \begin{equation}
    \label{Tiwari3.3ineq}
        d(T_1x,T_2y) \ge
       \alpha d(x,y) +
       \beta d(x,T_1x) + 
       \gamma d(y,T_2y) +
       \eta [d(x,T_1x) + d(y,T_2y) ]
    \end{equation}
    for all $x,y \in X$,
    where
    \[ \alpha \ge -1,~~~
    \beta > 0,~~~ \gamma \le 1/2,
    ~~~1/2 < \eta \le 1, ~~~\mbox{and}
    ~~~\alpha + \beta + \gamma + \eta > 1.
    \]
    Then
    $T_1$ and $T_2$ have a 
    common fixed point in~$X$.
\end{assert}

The argument offered as proof of
Assertion~\ref{Tiwari3.3} creates
a sequence $\{x_n\}_{n=1}^{\infty}$, 
reaches an inequality
\[ [1 - (\beta + \eta)]d(x_{2n},x_{2n+1})
\ge (\alpha + \gamma + \eta)d(x_{2n+1},x_{2n+2})
\]
and claims to derive that
%\begin{equation}
%\label{Tiwari3.3Pineq}
\[    d(x_{2n},x_{2n+1}) \ge
\frac{\alpha + \gamma + \eta}{1 - (\beta + \eta)} d(x_{2n+1},x_{2n+2}).
\]
%\end{equation} 
This reasoning would be correct if we
knew that $1 - (\beta + \eta) > 0$;
however, we don't have such knowledge,
since the hypotheses
allow $1 - (\beta + \eta) \le 0$. 

Thus, we must consider
Assertion~\ref{Tiwari3.3}, as written, unproven.

See section~\ref{modificationToTriviality}
for discussion of the triviality of this
assertion.

\subsection{\cite{TiwariEtAl}
``Theorem" 3.4}

The following is stated as Theorem~3.4
of~\cite{TiwariEtAl}.

\begin{assert}
  {\rm \cite{TiwariEtAl}}
    \label{Tiwari3.4}
    Let $(X,d,k)$ be a 
    complete digital metric 
    space and suppose
    $T_1,T_2: X \to X$ are
    continuous, onto mappings
    satisfying
    \begin{equation}
    \label{Tiwari3.4ineq}
        d(T_1x, T_2y) \ge
       \alpha [d(x,y) + d(x,T_1x) + d(y,T_2y)] +
       [\beta d(x,T_2y) + d(y,T_1x)]
    \end{equation}
    for all $x,y \in X$, where
    \[ \alpha \ge 0,~~ \beta <1,~~
       \alpha + \beta > 1.\]
    Then $T_1$ and $T_2$ have a common
    fixed point.
\end{assert}

It seems likely that
``$[\beta d(x,T_2y)$" is intended to
be ``$\beta [d(x,T_2y)$".

The ``proof" of Assertion~\ref{Tiwari3.4}
in~\cite{TiwariEtAl} contains errors
similar to those in the ``proof" 
of Assertion~\ref{Tiwari3.3}. We discuss one
of these errors.

The argument creates a sequence 
$\{x_n\}_{n=1}^{\infty}$,
reaches the inequality
\[ [1 - (\alpha + \beta)] d(x_{2n},x_{2n+1}) \ge
(2 \alpha + 2 \beta) d(x_{2n+1}, x_{2n+2})
\]
and claims to derive from it
\[ d(x_{2n},x_{2n+1}) \ge
  \frac{2 \alpha + 2 \beta}{1-(\alpha + \beta)}
  d(x_{2n+1},x_{2n+2})
\]
which does not follow, since by
hypothesis, the denominator of the
fraction is negative.

Thus we must regard
Assertion~\ref{Tiwari3.4} as
unproven.

See section~\ref{modificationToTriviality}
for discussion of the triviality of this
assertion.

\subsection{\cite{TiwariEtAl}
``Theorem" 3.5}
The following is stated as Theorem~3.5
of~\cite{TiwariEtAl}.

\begin{assert}
    \label{Tiwari3.5}
     Let $(X,d)$ be a complete metric 
    space and suppose
    $T_1,T_2: X \to X$ are continuous
    onto mappings
    satisfying \newline \newline
    $d(T_1x,T_2y) \ge$
    \begin{equation}
    \label{Tiwari3.5ineq}    
    \alpha \max \{ d(x,y), 
      d(x,T_1x), d(y,T_2y)\} +
      \beta \max \{ d(x,T_2y), d(x,y) \}
      + \gamma d(x,y)
     \end{equation}
    where 
    \[ \alpha \ge 0,~~~\beta > 0,~~~
       \gamma \le 1,~~~
       \alpha + \beta + \gamma > 1.
    \]
    Then $T_1$ and $T_2$ have a 
    common fixed point.
\end{assert}

The argument given as ``proof" of
Assertion~\ref{Tiwari3.5}
in~\cite{TiwariEtAl} is flawed as
follows. A sequence 
$\{x_n\}_{n=1}^{\infty} \subset X$
is created and the following
inequality is reached:
\begin{equation}
\label{Tiwari3.5Pineq}
    d(x_{2n+1},x_{2n+2}) \le 
   hd(x_{2n},x_{2n+1}),~~\mbox{where}~~
   h= \frac{1}{\alpha+\beta+\gamma}
\end{equation}
An implicit induction is then used
to claim that (\ref{Tiwari3.5ineq})
implies
\[ d(x_{2n},x_{2n+1}) \le h^{2n}d(x_0,x_1)
\]
However, the reasoning is incorrect,
since the left side 
of~ (\ref{Tiwari3.5Pineq}) requires
the smaller index to be odd, and there
is no analog for the smaller index
being even.

See section~\ref{modificationToTriviality}
for discussion of the triviality of this
assertion.

\subsection{On triviality of assertions
of~\cite{TiwariEtAl}}
\label{modificationToTriviality}
We consider conditions under which the assertions 
of~\cite{TiwariEtAl} reduce
to triviality. In the following, we require all
of the constants $\alpha, \beta, \gamma$, and
$\eta$ to be positive (perhaps this was intended
by the authors of~\cite{TiwariEtAl}, but as
written their assertions occasionally permit 
negative values). In other ways, our hypotheses
have greater generality, in that we omit certain 
hypotheses of~\cite{TiwariEtAl}.

\begin{prop}
        \label{Tiwari3.1trivialize}
    Let $(X,d)$ be a metric 
    space and suppose
    $T_1,T_2: X \to X$ are
    mappings such that $T_2$ is onto and
    for all $x,y \in X$,
   $T_1$ and $T_2$ satisfy any 
   of~(\ref{Tiwari3.1ineq}), (\ref{Tiwari3.2ineq}),
   (\ref{Tiwari3.3ineq}), (\ref{Tiwari3.4ineq}),
   or~(\ref{Tiwari3.5ineq})
    where all of $\alpha,\beta,\gamma$, and $\eta$
    are positive.
    Then $T_1 = T_2 = \id_X$.
    Further, if $\alpha > 1$ then
    $X$ has only one point.
\end{prop}

\begin{proof}
    Given $x_0 \in X$, $T_2$ being onto
    implies there exists $y_0 \in X$ such that
    $T_1 x_0 = T_2 y_0$. Thus, for the pair
    $(x_0,y_0)$, the left side each
    of~(\ref{Tiwari3.1ineq}), (\ref{Tiwari3.2ineq}),
   (\ref{Tiwari3.3ineq}), (\ref{Tiwari3.4ineq}),
   or~(\ref{Tiwari3.5ineq}) is 0, so each
    term of the right side is 0. Hence 
    \[ x_0 = y_0~~\mbox{and} ~~d(x_0,T_1 x_0) = 0 = d(y_0,T_2 y_0).\]
    Since $x_0$ is an arbitrary member of~$X$, it
    follows that $T_1 = T_2 = \id_X$.

    But then each
    of~(\ref{Tiwari3.1ineq}), (\ref{Tiwari3.2ineq}),
   (\ref{Tiwari3.3ineq}), (\ref{Tiwari3.4ineq}),
   or~(\ref{Tiwari3.5ineq}) implies
    \[ d(x,y) = d(T_1 x, T_2 y) \ge \alpha d(x,y),\]
    which is impossible if $x \neq y$ and 
    $\alpha > 1$. Thus $\alpha > 1$ implies $X$
    has a single point.
\end{proof}

\section{Further remarks}
We have continued the work
of~\cite{BxSt19,Bx19,Bx19-3,Bx20,Bx22,BxBad6}
in discussing flaws in some papers claiming
fixed point results in digital topology.
The literature of digital topology includes
many fixed point assertions that are correct, 
correctly proven, and beautiful. However,
papers we have considered here have many
errors and assertions that turn out to be
trivial.

Although authors are responsible for
their errors and other shortcomings,
it is clear that many of the papers
studied in the current paper were
reviewed inadequately, and should
have been rejected.

\end{document}